\documentclass[11pt,a4paper]{amsart}
\usepackage{amsmath,amssymb,amsthm,mathtools}
\usepackage{upgreek}
\usepackage{enumitem}
\usepackage{geometry}
\usepackage{xcolor}
\usepackage{hyperref}
\usepackage[numbers]{natbib}

\theoremstyle{plain}
\newtheorem{theorem}{Theorem}[section]
\newtheorem{lemma}[theorem]{Lemma}
\newtheorem{proposition}[theorem]{Proposition}
\newtheorem{proposition*}{Proposition}
\newtheorem{corollary}[theorem]{Corollary}

\theoremstyle{definition}

\newtheorem{example}{Example}

\theoremstyle{remark}
\newtheorem{remark}{Remark}[section]

\numberwithin{equation}{section}

\DeclareMathOperator{\E}{\mathbb{E}}
\DeclareMathOperator{\p}{\mathbb{P}}

\title[Catastrophe--dispersion processes]{Catastrophe--dispersion models in random and varying environments across generations}
\author{Lucas R. de Lima \and Fábio P. Machado}

\address{Department of Mathematics\\
Institute of Mathematical Sciences and Computation (ICMC)\\
University of S\~ao Paulo\\
Av. Trabalhador S\~ao-Carlense, 400\\
13566-590 S\~ao Carlos,  SP\\
Brazil}
\email{lrdelima@icmc.usp.br}

\address{Department of Statistics\\
Institute of Mathematics and Statistics (IME)\\
University of S\~ao Paulo\\
Rua do Mat\~ao, 1010\\
05508-090 S\~ao Paulo, SP\\
Brazil}
\email{fmachado@ime.usp.br}

\thanks{{\bf Funding:} Research supported by grants \#2023/13453-5 and \#2024/06021-4, S\~ao Paulo Research Foundation (FAPESP). This work was also supported by the National Institute of Science and Technology in Stochastic Modeling and Complexity (INCT-NUMEC), funded by CNPq (grant no. 408590/2024-6)}



\begin{document}

\begin{abstract}
We study a class of branching processes in which the offspring distribution is not specified directly but is induced by a cycle of internal colony growth, catastrophic reduction and structured dispersal. The parameters governing growth, survival and dispersal are allowed to vary deterministically or randomly from one generation to the next, giving rise to branching processes in varying and random environments with implicitly defined offspring laws. We show that survival and extinction are governed entirely by the associated log-mean process, exactly as in the classical theory. The paper treats four qualitatively different dispersal mechanisms and establishes a universal ordering of the induced offspring means. For Poissonian growth with binomial survival, explicit thresholds are obtained that determine extinction or survival uniformly over all four mechanisms. A series of ecologically motivated examples with Yule–Simon growth illustrates the versatility of the framework.
\end{abstract}

\maketitle

\section{Introduction}

Branching processes subject to catastrophic events have a long history in probability theory and applications, particularly in population dynamics,
epidemiology, and reliability theory.  In many classical models, catastrophes are represented as sudden population reductions followed by ordinary reproduction dynamics; see, for instance, the early work of Brockwell, Gani, and Resnick~\cite{brockwell_1982,brockwell_1986} as well as more recent contributions on binomial and geometric catastrophes \cite{artalejo_2007,kapodistria_2016,duque_2023,junior_2021,junior_2023}. A parallel line of research has shifted the focus to situations where catastrophes are followed by \emph{dispersion}: the surviving individuals do not remain together but instead found new colonies or communities according to some structured rule.  Such mechanisms arise naturally in ecology (metapopulations, fragmentation), epidemiology (local extinction and recolonization), and networked systems~\cite{lanchier_2024}.

The stochastic processes studied in this paper belong to this latter class. They are closely related to the catastrophe--dispersion models introduced and analysed in~\cite{junior_2016,machado_2017,machado_2018,schinazi_2015,amaya2025}, where a colony grows internally until a random catastrophe time, a random number of individuals survive, and the survivors disperse to found new colonies.  Our work broadens this programme by allowing the parameters that govern growth, catastrophe, survival, and dispersal to vary from one generation to the next, either deterministically or randomly.  (For a broad overview of related stochastic interacting systems we refer the reader to the recent monograph of Lanchier~\cite{lanchier_2024}.)

From a probabilistic point of view, the catastrophe–dispersion mechanism generates branching structures whose offspring distributions are \emph{not primitive}: they are not specified directly but are instead induced by the composition of internal growth, catastrophic survival, and dispersal.  This places the resulting colony‑count processes outside the standard setting of Bienaymé–Galton–Watson or Crump–Mode–Jagers processes, and forces us to verify that the classical theory of branching processes in varying and random environments still applies.

That theory was initiated by Jagers~\cite{jagers1974} for varying environments and by Smith and Wilkinson~\cite{smith_wilkinson_1969} for random environments; a comprehensive modern treatment is given in the monograph of Kersting and Vatutin~\cite{kersting_vatutin_2017}. The central insight is that the asymptotic behaviour—survival or extinction—is governed by the associated log‑mean process \(S_n = \sum_{k=0}^{n-1} \log \mu_k\), where \(\mu_k\) denotes the expected number of offspring in generation \(k\).  Our first goal is therefore to show that, once the \emph{induced offspring means} are identified, the catastrophe–dispersion model fits naturally into this framework and that its long‑term fate is determined by the sign (or the limsup/liminf) of the log‑mean sequence.

A distinctive feature of the paper is the systematic treatment of four qualitatively different dispersion mechanisms, ranging from complete dispersal (every survivor founds a new colony) to capacity‑limited and spatially structured dispersal.  These mechanisms coincide with those studied in~\cite{junior_2016,machado_2018}; they capture different ecological constraints and lead to different moment properties of the induced offspring distributions.  While the main classification depends only on the means \(\mu_n\), the dispersion rule becomes crucial for subcritical asymptotics and for the strength of selection effects.

\medskip

\noindent\textbf{Organization of the paper.}
Section~\ref{sec:model-def} introduces the catastrophe--dispersion model and the four dispersal mechanisms D1--D4.  Section~\ref{sec:environments} defines varying and random environments, the induced offspring means, and establishes the fundamental ordering the offspring means. Section~\ref{sec:classification} contains the survival--extinction classification.  After
recalling the classical criteria, we treat Poissonian growth with binomial survival and obtain explicit thresholds that hold uniformly for all four dispersal mechanisms.  A series of ecologically motivated examples with Yule--Simon growth completes the analysis.

\medskip

\section{Model and Notation} \label{sec:model-def}

The model studied here builds on the catastrophe--dispersion framework introduced in~\cite{junior_2016,machado_2017,machado_2018,schinazi_2015}, where a colony grows internally between catastrophic events, a random fraction of individuals survive, and the survivors disperse to found new colonies.  We extend this setup by allowing the parameters governing growth, catastrophes, survival, and dispersal to vary from one generation to the next, thereby embedding the process into the classical theory of branching processes in varying and random environments.  For completeness, we briefly recall the colony life cycle and the branching structure.

\subsection{Colony life cycle}

A \emph{colony} evolves according to the following steps.

\medskip
\noindent\textbf{Internal growth.}
Let \(\eta(t)\), \(t\ge 0\), denote the size of a single colony at time \(t\), started from one individual, \(\eta(0)=1\).
We consider two canonical internal growth mechanisms:
\begin{itemize}
    \item \textbf{Yule growth:} \(\eta(t)\) is a pure birth process with birth rate \(\lambda>0\).
    \item \textbf{Poissonian growth:} \(\eta(t)\) is a Poisson random variable with mean \(\lambda t\), used either as an approximation or as a phenomenological model.
\end{itemize}
All results in this paper are stated abstractly in terms of the law of \(\eta(T)\); explicit choices will be treated in subsequent sections and companion papers.

\medskip
\noindent\textbf{Catastrophe time.}
A catastrophe occurs at a random time
\[
T \sim \mathrm{Exp}(\theta),
\]
independent of the internal growth \(\eta(\cdot)\).
The parameter \(\theta>0\) represents the instantaneous rate of catastrophic events. In varying or random environments, \(\theta\) may change from one generation to the next. 

\begin{example}[Poissonian growth with exponential catastrophe time] \label{ex:poissonian_growth}
Let $\eta(t) = 1 + X(t)$, where $X(t) \sim \mathrm{Poi}(\lambda t)$ is a Poisson process, and let $T \sim \mathrm{Exp}(\theta)$ be independent of $N$. The probability generating function (pgf) of $\eta(T)$ is
\[
G_{\eta(T)}(z) := \mathbb{E}[z^{\eta(T)}] = \int_0^\infty \mathbb{E}[z^{\eta(T)} \mid T=t] \, \theta e^{-\theta t} \, dt.
\]

Conditional on $T=t$, we have
\[
\mathbb{E}[z^{\eta(T)} \mid T=t] = z \, \mathbb{E}[z^{X(t)}] = z \, e^{\lambda t (z-1)}.
\]

Hence,
\[
G_{\eta(T)}(z) = \int_0^\infty z \, e^{\lambda s (z-1)} \, \theta e^{-\theta s} \, ds = \frac{\theta z}{\theta + \lambda (1-z)}.
\]
This is the PGF of $\eta(T) \sim 1 + \mathrm{Geom}\!\left(\frac{\theta}{\theta+\lambda}\right)$ with support $\{1,2,3,\dots\}$.
\end{example}

\begin{example}[Yule-Simon growth]
Let $\eta(t)$ be a Yule process with birth rate $\lambda$, started from one individual, so that
\[
\mathbb P(\eta(t)=n)=e^{-\lambda t}(1-e^{-\lambda t})^{n-1},\qquad n\ge1.
\]
Let $T\sim \mathrm{Exp}(\theta)$ be independent of $\eta$. Conditional on $T=t$, we have
\[
\mathbb E[z^{\eta(T)}\mid T=t]
=
\frac{z e^{-\lambda t}}{1-(1-e^{-\lambda t})z}.
\]
Hence
\[
G_{\eta(T)}(z)
=
\int_0^\infty
\frac{z e^{-\lambda t}}{1-(1-e^{-\lambda t})z}\,
\theta e^{-\theta t}\,dt=
\frac{\theta}{\lambda}
\int_0^1
\frac{z\,u^{\theta/\lambda}}{1-(1-u)z}\,du.
\]
Using Euler's integral representation of the Gauss hypergeometric function,
\[
G_{\eta(T)}(z)
=
z\frac{\theta/\lambda}{1+\theta/\lambda}
\,{}_2F_1\!\left(1,1;2+\theta/\lambda;z\right),
\qquad |z|<1.
\]
In particular, $\eta(T)$ has a Yule--Simon distribution with parameter $\rho=\theta/\lambda$; consequently $\mathbb P(\eta(T)\geq n) = \mathcal{O}(n^{-\theta/\lambda})$. Therefore, for Yule growth with birth rate $\lambda$ observed at an independent exponential time $T\sim\mathrm{Exp}(\theta)$, one has
\[
\mathbb E[\eta(T)] = \frac{\theta}{\theta-\lambda}\quad\text{if }\theta>\lambda,
\]
and $\mathbb E[\eta(T)]=+\infty$ otherwise.
\end{example}
\medskip
\noindent\textbf{Survivors.}
Let \(N_{\eta(T)}\) denote the number of survivors immediately after the catastrophe. Conditionally on \(\eta(T)=m\), the law of \(N_m\) is allowed to depend on the environment and models the severity of the catastrophe. Throughout, we assume only that
\[
0 \le N_m \le m \quad \text{a.s.},
\]
and that the conditional law of \(N_m\) is stochastically non-decreasing in \(m\). We list below some typical examples:
\begin{itemize}
    \item \textbf{Uniform catastrophes.:} The number of survivors of the catastrophe is uniformly distributed between $0$ and $m-1$, \textit{i.e.},  $N_m \sim \mathrm{Unif}_{\{0,1, \dots, m-1\}}$.

    \item \textbf{Binomial catastrophes.:} The classical model of catastrophe that considers that, independently, each individual survives the catastrophe with probability $p$. Therefore, $N_m \sim \mathrm{Bin}(m,p)$.

    \item \textbf{Geometric catastrophes.:} In this case, the catastrophe hits the individuals sequentially and independently, the catastrophic event ceases with probability $p$ when affecting each individual. We model it considering a random variable $G \sim \mathrm{Geom}(p)$ and $N_m := (m-G)^+$.
\end{itemize}
\medskip
\noindent\textbf{Dispersion and founding of new colonies.}
The survivors disperse and establish new colonies according to a dispersion rule. Given \(r\) \textit{survivors} (the \textit{remaining population}), the number of newly founded colonies is
\[
\Delta(r),
\]
where \(\Delta\) is a dispersion operator taking values in $\{0,1,\dots, r\}$. In some dispersion models, an environmental resource parameter \(d\) is considered so that $\Delta(r)\le d$. This operator encodes the ecological constraints and spatial structure of dispersal.

\subsection{Dispersion mechanisms}

The paper focuses on four dispersion mechanisms, which coincide with
those introduced and studied in~\cite{junior_2016,machado_2017,machado_2018,schinazi_2015}. We distinguish between two qualitatively different regimes. Under (D1), each surviving individual independently founds a new colony, with no territorial constraint. In contrast, under (D2)--(D4) dispersion is territory-limited and colonies are founded by single individuals: each available territory can host at most one founder, inducing an additional layer of competition after the catastrophe.

\begin{description}
    \item[(D1) Full dispersal.]
    Every survivor founds a new colony at a distinct site:
    \[
    \Delta^{(1)}(r) = r.
    \]

    \item[(D2) Capacity-limited dispersal.]
    At most \(d\) new colonies can be founded:
    \[
    \Delta^{(2)}(r) = \min\{d,r\}.
    \]

    \item[(D3) Independent site occupation.]
    There are $d$ available sites. Conditionally on $N_{\eta(T)}=r$, each surviving individual independently chooses one of the $d$ sites with equal probability $1/d$. Let $r_j$ denote the number of individuals assigned to site $j$, so that $r=\sum_{j=1}^d r_j$ and
    \[
        (r_1,\dots,r_d)\sim \mathrm{Multinomial}\bigl(r;1/d,\dots,1/d\bigr).
    \] 
    A new colony is founded at a site if and only if it receives at least one individual.
    The number of founded colonies is
    \[
        \Delta^{(3)}(r) := \sum_{j=1}^d \mathbf 1_{\{r_j\ge1\}}.
    \]
    Equivalently, conditional on $r$,
    \[
        \mathbb P\left(\Delta^{(3)}(r)=k\right) = \binom dk \sum_{i=0}^k (-1)^i \binom ki \Bigl(\frac{k-i}{d}\Bigr)^{r}, \qquad k=0,\dots,d.
    \]

    \item[(D4) Uniform dispersal.]
    Given $r\in\mathbb N_0$ surviving individuals and $d\in\mathbb N$ available territories, individuals are distributed uniformly at random among the $d$ territories, that is, all occupancy vectors $(r_1,\dots,r_d)\in\mathbb N_0^d$ satisfying \[ r_1+\cdots+r_d = r\]are equally likely (stars and bars model). At most one colony can be founded per territory, and a territory is successfully colonized if and only if it receives at least one individual.

    The number of founded colonies $\Delta^{(4)}(r)$ then satisfies, for $k=1,\dots,\min\{d,r\}$,
    \[
    \mathbb P\bigl(\Delta^{(4)}(r)=k\bigr) =\frac{\binom{d}{k}\binom{r-1}{k-1}}{\binom{r+d-1}{r}}.
    \]
    Moreover, $\mathbb P\bigl(\Delta^{(4)}(r)=0\bigr)=\mathbf{1}_{{r=0}}$.
\end{description}
\medskip

These four mechanisms represent qualitatively different dispersal regimes: full dispersal (D1), hard resource limitation (D2), spatial competition (D3), and uniform dispersion effects (D4). Observe that the dispersal mechanisms D2--D4 are equivalent when $d=1$.

\begin{lemma}[Expectation under independent site occupation $D3$]
\label{lem:mean-D3}
Let $d\ge1$ and let $Y:=N_{\eta(T)}$ be an $\mathbb N_0$-valued random variable with probability generating function
\[
G_Y(s)=\mathbb E[s^Y], \qquad |s|<1.
\]
Under dispersal mechanism $D3$, the expected number of founded colonies is
\[
\mathbb E\!\left[\Delta^{(3)}(Y)\right]
=
d\left(1-G_Y\!\left(1-\frac1d\right)\right).
\]
\end{lemma}

\begin{proof}
Fix $r\in\mathbb N_0$. Conditionally on $Y=r$, each of the $r$ surviving individuals independently chooses one of the $d$ sites with probability $1/d$ each.

For $j\in\{1,\dots,d\}$, define
\[
I_j:=\mathbf 1_{\{r_j\ge1\}},
\]
where $r_j$ is the number of individuals assigned to site $j$. Then
\[
\Delta^{(3)}(r)=\sum_{j=1}^d I_j.
\]

By linearity of expectation and symmetry, $\mathbb E[\Delta^{(3)}(r)]=\sum_{j=1}^d \mathbb E[I_j]=d\,\mathbb E[I_1]$. Now, $\mathbb E[I_1]=\mathbb P(r_1\ge1)=1-\mathbb P(r_1=0)$. The event $\{r_1=0\}$ means that none of the $r$ individuals chooses site $1$. Since choices are independent,
\[
\mathbb P(r_1=0)=\left(1-\frac1d\right)^r.
\]
Hence $\mathbb E[I_1]=1-\left(1-\frac1d\right)^r$, and therefore
\[
\mathbb E[\Delta^{(3)}(r)]
=
d\left(1-\left(1-\frac1d\right)^r\right).
\]
Thus
\[
\mathbb E[\Delta^{(3)}(Y)]
=
d\left(1-\sum_{r\ge0}\mathbb P(Y=r)\left(1-\frac1d\right)^r\right).
\]
We conclude that
\[
\mathbb E\!\left[\Delta^{(3)}(Y)\right]
=
d\left(1-G_Y\!\left(1-\frac1d\right)\right).
\]
This completes the proof.
\end{proof}

\begin{lemma}[Expectation under uniform dispersal $D4$]
\label{lem:mean-D4}
Let $d\ge1$ and let $Y:=N_{\eta(T)}$ be an $\mathbb N_0$-valued random variable with probability generating function
\[
G_Y(z)=\mathbb E[z^Y], \qquad 0\le z\le1.
\]
Under dispersal mechanism $D4$, the expected number of founded colonies satisfies
\[
\mathbb E\!\left[\Delta^{(4)}(Y)\right]
=
\mathbb E\!\left[\frac{dY}{Y+d-1}\mathbf 1_{\{Y\ge1\}}\right].
\]
In particular, if $d\ge2$, then
\[
\mathbb E\!\left[\Delta^{(4)}(Y)\right]
=
d
-
d(d-1)\int_0^1 z^{d-2}G_Y(z)\,dz.
\]
\end{lemma}

\begin{proof}
Fix $r\in\mathbb N_0$. Under mechanism $D4$, all occupancy vectors
\[
(r_1,\dots,r_d)\in\mathbb N_0^d,\qquad r_1+\cdots+r_d=r,
\]
are equally likely. The number of founded colonies is precisely the number of occupied territories:
\[
\Delta^{(4)}(r)=\sum_{i=1}^d \mathbf 1_{\{r_i\ge1\}}.
\]

By symmetry,
\[
\mathbb E[\Delta^{(4)}(r)]
=
d\,\mathbb P(r_1\ge1).
\]

Now the total number of admissible occupancy vectors is $\binom{r+d-1}{r}$. Those with $r_1=0$ correspond to distributing $r$ individuals among the remaining $d-1$ territories, hence there are $\binom{r+d-2}{r}$ such vectors. Therefore,
\[
\mathbb P(r_1=0)
=
\frac{\binom{r+d-2}{r}}{\binom{r+d-1}{r}}
=
\frac{d-1}{r+d-1},
\]
and so
\[
\mathbb P(r_1\ge1)
=
1-\frac{d-1}{r+d-1}
=
\frac{r}{r+d-1}.
\]
Hence, for $r\ge1$,
\[
\mathbb E[\Delta^{(4)}(r)]
=
\frac{dr}{r+d-1},
\]
while for $r=0$ clearly $\Delta^{(4)}(0)=0$. Thus $\mathbb E[\Delta^{(4)}(r)]=\frac{dr}{r+d-1}\mathbf 1_{\{r\ge1\}}$.

Conditioning on $Y$, one has
\[
\mathbb E[\Delta^{(4)}(Y)]
=
\sum_{r\ge0}\mathbb P(Y=r)\,\mathbb E[\Delta^{(4)}(r)]
=
\mathbb E\!\left[\frac{dY}{Y+d-1}\mathbf 1_{\{Y\ge1\}}\right].
\]

Assume now $d\ge2$. For $r\ge1$, $\frac{dr}{r+d-1}=d-\frac{d(d-1)}{r+d-1}$.Therefore,
\[
\mathbb E[\Delta^{(4)}(Y)]
=
d\,\mathbb P(Y\ge1)
-
d(d-1)\sum_{r\ge1}\frac{\mathbb P(Y=r)}{r+d-1}.
\]

Using $\frac1{r+d-1}=\int_0^1 z^{r+d-2}\,dz$, Tonelli's theorem yields
\[
\sum_{r\ge1}\frac{\mathbb P(Y=r)}{r+d-1}
=
\int_0^1 z^{d-2}\sum_{r\ge1}\mathbb P(Y=r)z^r\,dz=\int_0^1 z^{d-2}\bigl(G_Y(z)-\mathbb P(Y=0)\bigr)\,dz.
\]
Substituting into the previous identity proves the result.
\end{proof}

\subsection{Colony-level branching structure}\label{sec:branching-structure}

Colonies reproduce only through catastrophes and dispersal. Let \(Z_n\) denote the number of colonies in generation \(n\).
Each colony in generation \(n\) evolves independently and produces a random number of offspring colonies.

For colony \(i\) in generation \(n\), let
\[
\xi_{i,n}
\;\overset{d}{=}\;
\Delta\!\left(N_{\eta(T)}\right)
\]
denote the number of colonies founded by that colony. The colony-count process satisfies the recursion
\[
Z_{n+1}
\;=\;
\sum_{i=1}^{Z_n} \xi_{i,n}.
\]
Conditionally on the environment, the random variables \(\{\xi_{i,n}\}_{i\ge1}\) are independent and identically distributed.

\section{Varying and Random Environments} \label{sec:environments}

The reproduction of colonies is governed by a generation--dependent environment. At generation \(n\), the environment is described by a vector
\[
\mathcal{E}_n = (\theta_n,\lambda_n,p_n,d_n),
\]
where \(\theta_n\) denotes the catastrophe rate, \(\lambda_n\) the internal growth parameter, \(p_n\) parametrizes the survivor kernel governing post–catastrophe survival, and \(d_n\) the resource or territorial constraint acting at dispersion.

The environment affects reproduction only through the induced offspring distribution. Given \(\mathcal{E}_n\), a single colony produces a random number of offspring colonies \(\xi_{i,n}\), whose law is determined by the catastrophe--dispersion mechanism under \(\mathcal{E}_n\). We define the associated offspring mean by
\[
\mu_n =\mu(\mathcal E_n)
\;:=\;
\mathbb{E}\!\left[
\Delta\!\left(N_{\eta(T)}\right)
\;\middle|\;
\mathcal{E}_n
\right].
\]

The colony--count process \((Z_n)_{n\ge0}\) is thus a branching process whose reproduction law in generation \(n\) has mean \(\mu_n\). The corresponding log--mean sequence is
\[
\ell_n := \log \mu_n,
\qquad
S_n := \sum_{k=0}^{n-1} \ell_k.
\]

The asymptotic behavior of \((S_n)_{n\ge0}\) governs extinction, survival, and finer asymptotic regimes, in accordance with the general theory of branching processes in varying and random environments.

The nature of the environmental sequence \((\mathcal{E}_n)_{n\ge0}\) determines the regime under consideration. If the sequence is deterministic, the process is said to evolve in a \emph{varying environment}. If the sequence is random, the process is said to evolve in a \emph{random environment}; unless stated otherwise, the sequence is assumed to be stationary and ergodic, with the i.i.d.\ case as a particular instance.

To emphasize the dispersal mechanism, we write $\mu_n^{(i)}=\mu^{(i)}(\mathcal E_n):=\E[\Delta^{(i)}(N_{\eta(T)})\mid\mathcal E_n]$. The same notation is used for $\ell_n^{(i)}$ and $S_n^{(i)}$. 

\begin{lemma}[Ordering of the induced offspring means]
\label{lem:ordering}
Let $\mathcal{E}= (\theta, \lambda,p,d)$ be an environment, and  let \(Y = N_{\eta(T)}\) be an \(\mathbb N_{0}\)-valued random variable representing the number of survivors after a catastrophe.
Then, for every distribution of \(Y\) and every \(d\ge 1\),
\[
\mu^{(4)}(\mathcal E) \;\le\; \mu^{(3)}(\mathcal E) \;\le\; \mu^{(2)}(\mathcal E) \;\le\; \mu^{(1)}(\mathcal E) .
\]
\end{lemma}

\begin{proof}
Since \(\Delta^{(3)}(Y) \leq\min(d,Y)= \Delta^{(2)}(Y) \le  Y =\Delta^{(1)}(Y)\) almost surely by definition, then \(\mu^{(3)}(\mathcal E)\le\mu^{(2)}(\mathcal E)\le\mu^{(1)}(\mathcal E)\). Therefore, it suffices to verify that \(\mu^{(4)}(\mathcal E)\le\mu^{(3)}(\mathcal E)\).

By Lemmas~\ref{lem:mean-D3} and \ref{lem:mean-D4}, we have
\[
\mu^{(3)}(\mathcal E) = d\left(1 - \mathbb E\left.\left[(1-\tfrac1d)^{Y}\right| \mathcal E\right]\right),\quad \text{and} \quad
\mu^{(4)}(\mathcal E) = \mathbb E\!\left.\left[ \frac{dY}{Y+d-1}\mathbf 1_{\{Y\ge 1\}}\right| \mathcal E \right].
\]

If \(d=1\), then \(\mu^{(4)}(\mathcal E) = \mathbb P(Y\ge 1\mid\mathcal E)=\mu^{(3)}(\mathcal E)\) and the inequality holds with equality.
Assume now \(d\ge 2\) and let \(r\ge 1\). The conditional expectations satisfy
\[
\frac{dr}{r+d-1} \le d\bigl(1-(1-\tfrac1d)^{r}\bigr)
\;\Longleftrightarrow\;
\Bigl(1-\frac1d\Bigr)^{\!r} \le \frac{d-1}{r+d-1}.
\]
Fix \(n = d-1\ge 1\); the right‑hand side becomes \(\frac{n}{n+r}\) and the left‑hand side equals \(\bigl(\frac{n}{n+1}\bigr)^{r}\).
For \(r=1\) or \(0\), both sides  are equal. For \(r\ge 2\), Bernoulli's inequality applied to \((1+\frac1n)^{r}\) yields
\[
(n+1)^{r} = n^{r}\Bigl(1+\frac1n\Bigr)^{\!r}
\ge n^{r}\Bigl(1+\frac{r}{n}\Bigr)
= n^{r} + r n^{r-1}
= n^{r-1}(n+r).
\]
Dividing by \((n+1)^{r}(n+r)\) gives
\[
\Bigl(\frac{n}{n+1}\Bigr)^{\!r} \le \frac{n}{n+r},
\]
which is exactly the desired inequality. Taking expectations over \(Y\) gives \(\mu^{(4)}(\mathcal E)\le\mu^{(3)}(\mathcal E)\).
\end{proof}

\section{Classification of extinction and survival}
\label{sec:classification}

In this section we apply the classical theory of branching processes in varying and random environments to the catastrophe--dispersion model. The key novelty is that the offspring distributions are not given directly but are induced by the composition of internal growth, catastrophic survival, and dispersal.  Once the offspring means are identified, the survival--extinction
classification depends only on the associated log‑mean process.

Throughout we assume that the environment sequence $(\mathcal E_n)_{n\ge0}$ is deterministic (varying environment) or stationary and ergodic (random environment), with the i.i.d.\ case as a particular instance. The following two theorems recall the classical criteria. They are special cases of results by Jagers~\cite{jagers1974} for varying environments and by Athreya and Karlin~\cite{athreya1971} for stationary ergodic random environments; the i.i.d.\ case goes back to Smith and Wilkinson~\cite{smith_wilkinson_1969}.  A comprehensive modern treatment is given in Kersting--Vatutin~\cite{kersting_vatutin_2017}.

\begin{theorem}[Varying environment]
\label{thm:VE-classification}
Assume $(\mathcal E_n)_{n\ge0}$ is deterministic and $0<\mu_n<+\infty$ for all $n$.
Define
\[
\underline{\Lambda}=\liminf_{n\to\infty}\frac{S_n}{n},\qquad
\overline{\Lambda}=\limsup_{n\to\infty}\frac{S_n}{n}.
\]
\begin{itemize}
  \item[(a)] If $\overline{\Lambda}<0$, then $\p(Z_n\to0)=1$ (extinction).
  \item[(b)] If $\underline{\Lambda}>0$, then $\p(Z_n\to\infty)>0$ (survival with positive probability).
\end{itemize}
\end{theorem}

\begin{theorem}[Random environment]
\label{thm:RE-classification}
Let $(\mathcal E_n)_{n\ge0}$ be a stationary ergodic sequence.
Assume that
\[
\E[|\ell_0|]<\infty \qquad\text{and}\qquad
\E\!\Big[\bigl|\log\bigl(\p(\xi_{1,0}>0\mid\mathcal E_0)\bigr)\bigr|\Big]<+\infty,
\]
where $\xi_{1,0}$ denotes the number of offspring colonies of a single
colony given the environment $\mathcal E_0$ (see
Section~\ref{sec:branching-structure}).
Set $\Lambda=\E[\ell_0]$. Then:
\begin{itemize}
  \item[(a)] If $\Lambda>0$, survival has positive probability.
  \item[(b)] If $\Lambda<0$, extinction is almost sure.
  \item[(c)] If $\Lambda=0$ and the offspring law is non‑degenerate,
        extinction is almost sure.
\end{itemize}
\end{theorem}
\begin{remark}
For i.i.d.\ environments the same criteria hold without stating the extra condition on $\p(\xi_{1,0}>0\mid\mathcal E_0)$; see \cite{smith_wilkinson_1969} and \cite{kersting_vatutin_2017}. The extension to stationary ergodic environments is due to Athreya and Karlin~\cite{athreya1971} (Theorems~1 and~3), whose formulation we follow here.
\end{remark}

The remainder of this section translates the abstract criteria into explicit, computable thresholds for two canonical choices of internal growth. First, in Section \ref{sec:poisson-binomial-explicit} we work with Poissonian growth. Because the offspring means can be expressed in closed form for every dispersion mechanism, the ordering established in Lemma~\ref{lem:ordering} yields
inequalities that govern extinction or survival \emph{uniformly} over all four dispersal mechanisms D1--D4.  The results of this subsection are therefore structural: they show how the interplay of growth, catastrophe, and dispersal determines the fate of the population in a way that is largely insensitive to the precise dispersion rule.

In Section \ref{sec:yule-examples} we turn to Yule--Simon growth. Here explicit formulas for $\mu^{(2)}$--$\mu^{(4)}$ are more involved, and we concentrate on full dispersal (D1) to illustrate a range of ecologically motivated scenarios: decreasing catastrophe intensity, cumulative fertility loss, correlated random catastrophes, and adaptive survival. These examples highlight the versatility of the framework and demonstrate how the log‑mean criterion can be applied in concrete settings of biological interest.

\subsection{Thresholds for models of Poissonian growth with binomial survival}
\label{sec:poisson-binomial-explicit}

Throughout this subsection we work with the following model: \begin{example}[Poissonian growth with binomial catastrophe]
\label{ex:poisson-binomial}
Let the internal growth be $\eta(t)=1+X(t)$, where $\big(X(t)\big)_{t\ge0}$ is a
Poisson process of rate $\lambda>0$.  
The catastrophe time is $T\sim\mathrm{Exp}(\theta)$, $\theta>0$, independent of
$\eta$ (see Example~\ref{ex:poissonian_growth}).  Conditionally on $\eta(T)=m$, each individual survives independently
with probability $p\in(0,1]$, so that
\[
Y := N_{\eta(T)}\sim\mathrm{Bin}(m,p),\qquad m\ge1.
\]
The dispersion mechanism is one of D1--D4, with site parameter $d\in\mathbb N$. Set $\mathcal{E} = (\theta,\lambda,p, d)$ to be the environment. For this model the probability generating function of $Y$ is
\[
G_{Y\mid\mathcal{E}}(z)=\frac{(1-q)(1-p+pz)}{1-q(1-p+pz)},
\qquad q:=\frac{\lambda}{\theta+\lambda},\;0\le z\le1,
\]
and the offspring means are
\[
\mu^{(1)}(\mathcal{E}) = p\Bigl(1+\frac{\lambda}{\theta}\Bigr),\qquad
\mu^{(2)}(\mathcal{E}) = \E[\min\{d,Y\}\mid \mathcal{E}],\qquad
\mu^{(3)}(\mathcal{E}) = d\Bigl(1-G_{Y\mid\mathcal{E}}\!\bigl(1-\tfrac1d\bigr)\Bigr),
\]
\[
\mu^{(4)}(\mathcal{E}) = \E\!\left.\left[\frac{dY}{Y+d-1}\mathbf 1_{\{Y\ge1\}}\right|\mathcal{E}\right]
        = d - d(d-1)\!\int_0^1 \frac{z^{d-2}(1-p+pz)}{1+\frac{\lambda p}{\theta}
          -\frac{\lambda p}{\theta}z}\,dz .
\]
When $d=1$, $\mu^{(4)}=\mu^{(3)}=\mu^{(2)}$.
\end{example}

\begin{theorem}[Varying environment: convergence of parameters]
\label{thm:VE-poisson-binomial}
Consider the Poissonian growth with binomial survival of
Example~\ref{ex:poisson-binomial}. Let $(\mathcal E_n)_{n\ge0}$ be a varying environment sequence with
components $(\theta_n,\lambda_n,p_n,d_n)$. Assume
\[
\lim_{n\to\infty}(\theta_n,\lambda_n,p_n)=(\theta,\lambda,p)\in(0,\infty)^3,
\qquad \theta>\lambda,\ p\le1,
\]
and that $d_n\to+\infty$.  Then, for every dispersion mechanism $\Delta^{(j)}$,
$j=1,2,3,4$,
\begin{itemize}
  \item if $p>\dfrac{\theta}{\theta+\lambda}$, the process is supercritical and survives with positive probability;
  \item if $p<\dfrac{\theta}{\theta+\lambda}$, the process is subcritical and becomes extinct almost surely.
\end{itemize}
\end{theorem}

\begin{proof}
For the full dispersal mechanism (D1),
\[
\mu_n^{(1)} = p_n\Bigl(1+\frac{\lambda_n}{\theta_n}\Bigr),
\]
which by continuity converges to $\mu:=p(1+\lambda/\theta)$ as $n\to\infty$. Hence, the Cesàro means are so that $\frac1n S_n^{(1)}\to\log\mu$.

Now consider mechanism D4.  For any fixed integer $r\ge0$, the definition of $\Delta^{(4)}$ implies that, conditionally on $Y_n=r$, $\Delta^{(4)}(Y_n)$ is the number of occupied territories among $d_n$ when $r$ individuals are uniformly placed.  As $d_n\to\infty$, the probability that two individuals share a territory tends to $0$, so $\Delta^{(4)}(Y_n)\to r$ almost surely.
Moreover $\Delta^{(4)}(Y_n)\le Y_n$ almost surely.  Since $Y_n$ can be coupled with a geometric random variable whose mean is uniformly bounded (because $\theta_n$ eventually exceeds $\lambda_n+\varepsilon$ for some $\varepsilon>0$),
the family $\{Y_n\}$ is uniformly integrable, and consequently $\{\Delta^{(4)}(Y_n)\}$ is uniformly integrable as well.  Thus
\[
\mu_n^{(4)} = \E[\Delta^{(4)}(Y_n)] \to \E[Y_n] = \mu .
\]
Therefore $\frac1n S_n^{(4)}\to\log\mu$. By Lemma~\ref{lem:ordering}, for every $n$ and $j=1,2,3,4$,
\[
\mu_n^{(4)}\le\mu_n^{(j)}\le\mu_n^{(1)} .
\]
Taking logarithms and averaging gives
\[
\frac{S_n^{(4)}}{n} \le \frac{S_n^{(j)}}{n} \le \frac{S_n^{(1)}}{n}.
\]
Since the lower and upper bounds both converge to $\log\mu$, the squeeze theorem yields $\frac1n S_n^{(j)}\to\log\mu$ for all $j$.

Now apply Theorem~\ref{thm:VE-classification}.  The process is supercritical (survival with positive probability) precisely when $\log\mu>0$, i.e. $\mu=p(1+\lambda/\theta)>1$, which is equivalent to $p>\theta/(\theta+\lambda)$.  It is subcritical (extinction almost surely) when $\log\mu<0$, i.e. $p<\theta/(\theta+\lambda)$.
\end{proof}

\begin{theorem}[Varying environment: interval bounds]
\label{thm:VE-interval-bounds} Consider the Poissonian growth with binomial survival of Example~\ref{ex:poisson-binomial}. Let $(\mathcal E_n)_{n\ge0}$ be a deterministic environment sequence whose components satisfy, for all sufficiently large $n$,
\[
\theta_n\in[\theta_-,\theta_+],\quad
\lambda_n\in[\lambda_-,\lambda_+],\quad
p_n\in[p_-,p_+],\quad
d_n\in[d_-,d_+],
\]
with $0<\lambda_-\le\lambda_+<\theta_-\le\theta_+<+\infty$, $0<p_-\le p_+\le1$, and integers $1\le d_-\le d_+$. Then, for every dispersion mechanism $\Delta^{(j)}$, $j=1,2,3,4$:
\begin{itemize}
  \item if
  $\displaystyle p_+<\frac{\theta_-}{\theta_-+\lambda_+}$,
  the process becomes extinct almost surely;
  \item if
  \[
    \frac{d_-}{d_--1}\,\frac{1-p_-}{p_-}\,
    {}_2F_1\!\Bigl(1,1;d_-;-\frac{\lambda_-p_-}{\theta_+}\Bigr) + {}_2F_1\!\Bigl(1,1;d_-+1;-\frac{\lambda_-p_-}{\theta_+}\Bigr)< \frac{1}{p_-},
  \]
  where ${}_2F_1$ is the hypergeometric function, then the process survives with positive probability.
\end{itemize}
\end{theorem}

\begin{proof}
Consider the Poissonian‑binomial model of Example~\ref{ex:poisson-binomial}. For a given environment $\mathcal{E}=(\theta,\lambda,p,d)$, the offspring means are $\mu^{(1)}(\mathcal{E})=p(1+\lambda/\theta)$ and, by Lemma~\ref{lem:mean-D4} together with the hypergeometric representation derived in
Example~\ref{ex:poisson-binomial},
\[
\mu^{(4)}(\mathcal{E})=d-d(1-p)\,{}_2F_1\!\Bigl(1,1;d;-\frac{\lambda p}{\theta}\Bigr)
           -p(d-1)\,{}_2F_1\!\Bigl(1,1;d+1;-\frac{\lambda p}{\theta}\Bigr).
\]

Both $\mu^{(1)}(\mathcal{E})$ and $\mu^{(4)}(\mathcal{E})$ are monotone in the environment parameters:
$\mu^{(1)}(\mathcal{E})$ is increasing in $p$ and $\lambda$ and decreasing in $\theta$; $\mu^{(4)}(\mathcal{E})$ is increasing in $p,\lambda,d$ and decreasing in $\theta$, which follows from stochastic domination of the survivor variable $Y$ and the monotonicity of $\Delta^{(4)}$. Consequently, for all $n$ large enough,
\[
\mu_n^{(1)}\le \mu^{(1)}(\theta_-,\lambda_+,p_+)
           =p_+\Bigl(1+\frac{\lambda_+}{\theta_-}\Bigr)=:M_1,
\]
\[
\mu_n^{(4)}\ge \mu^{(4)}(\theta_+,\lambda_-,p_-,d_-)=:M_4 .
\]

Lemma~\ref{lem:ordering} yields
$\mu_n^{(4)}\le\mu_n^{(j)}\le\mu_n^{(1)}$
for $j=2,3$, and therefore
\[
\frac{S_n^{(4)}}{n}\le\frac{S_n^{(j)}}{n}\le\frac{S_n^{(1)}}{n}.
\]

The condition $p_+<\theta_-/(\theta_-+\lambda_+)$ is equivalent to $M_1<1$. Thus, for large $n$, $\mu_n^{(1)}\le M_1<1$, whence $\ell_n^{(1)}\le\log M_1<0$ and $\limsup_n S_n^{(1)}/n\le\log M_1<0$. The upper bound on $S_n^{(j)}/n$ implies $\limsup_n S_n^{(j)}/n<0$ for every $j$, and Theorem~\ref{thm:VE-classification} forces almost sure extinction.

A direct algebraic manipulation shows that the inequality
\[
\frac{1}{p_-}-{}_2F_1\!\Bigl(1,1;d_-+1;-\frac{\lambda_-p_-}{\theta_+}\Bigr)
 >\frac{d_-}{d_--1}\,\frac{1-p_-}{p_-}\,
   {}_2F_1\!\Bigl(1,1;d_-;-\frac{\lambda_-p_-}{\theta_+}\Bigr)
\]
is exactly equivalent to $M_4>1$. Hence, for all large $n$, $\mu_n^{(4)}\ge M_4>1$, so $\ell_n^{(4)}\ge\log M_4>0$ and
$\liminf_n S_n^{(4)}/n\ge\log M_4>0$. Since $S_n^{(j)}\ge S_n^{(4)}$, every $\liminf_n S_n^{(j)}/n$ is also strictly positive, and Theorem~\ref{thm:VE-classification} guarantees survival with positive probability for all dispersal mechanisms.
\end{proof}

\begin{corollary}[Sufficient conditions for universal survival]
\label{cor:VE-survival-sufficient}
Under the hypotheses of Theorem~\ref{thm:VE-interval-bounds}, if
\[
d_->\frac1{p_-}\qquad\text{and}\qquad
\frac{\lambda_-}{\theta_+}>\frac1{(d_--1)p_-},
\]
then the catastrophe–dispersion process survives with positive probability for every dispersion mechanism \(\Delta^{(j)}\), \(j=1,2,3,4\).
\end{corollary}

\begin{proof}
From the proof of Theorem~\ref{thm:VE-interval-bounds}, survival for all mechanisms is guaranteed as soon as \(\mu^{(4)}_{\min}>1\), where
\[
\mu^{(4)}_{\min}=d_--d_-(d_--1)\!\int_0^1
\frac{(1-p_-)t^{d_--2}+p_-t^{d_--1}}{1+at}\,dt,
\qquad a=\frac{\lambda_-p_-}{\theta_+}.
\]
Define \(I(p,a)=\int_0^1\frac{(1-p)t^{d_--2}+p\,t^{d_--1}}{1+at}\,dt\). The numerator satisfies
\[
\frac{\partial}{\partial p}\bigl[(1-p)t^{d_--2}+p\,t^{d_--1}\bigr]
= t^{d_--2}(t-1)\le0\qquad(0\le t\le1),
\]
so the integrand is pointwise non‑increasing in \(p\); hence \(I(p,a)\) is non‑increasing in \(p\).  Moreover, \(I(p,a)\) is strictly decreasing in \(a\). The hypotheses \(d_->1/p_-\) and \(\frac{\lambda_-}{\theta_+}>\frac1{(d_--1)p_-}\) are exactly \(p_->1/d_-\) and \(a>1/(d_--1)\).  Consequently,
\[
I(p_-,a)<I(1/d_-,1/(d_--1)).
\]
At the boundary \(p=1/d_-\), \(a=1/(d_--1)\) the integrand simplifies dramatically:
\[
\frac{(1-p)t^{d_--2}+p\,t^{d_--1}}{1+at}
 =\frac{\frac{d_--1}{d_-}t^{d_--2}+\frac1{d_-}t^{d_--1}}
        {1+\frac{t}{d_--1}}
 =\frac{d_--1}{d_-}\,t^{d_--2}.
\]
Thus
\[
I(1/d_-,1/(d_--1))=\frac{d_--1}{d_-}\int_0^1 t^{d_--2}dt
 =\frac1{d_-}.
\]
Therefore \(I(p_-,a)<1/d_-\), and
\[
\mu^{(4)}_{\min}=d_--d_-(d_--1)\,I(p_-,a)
 >d_--d_-(d_--1)\frac1{d_-}=1.
\]
The claim follows from Theorem~\ref{thm:VE-interval-bounds} (survival for all mechanisms whenever \(\mu^{(4)}_{\min}>1\)).
\end{proof}

\begin{theorem}[Random environment: simple sufficient conditions]
\label{thm:RE-poisson-binomial}
Consider the Poissonian growth with binomial survival of Example~\ref{ex:poisson-binomial}.  Let $(\mathcal E_n)_{n\ge0}$ be a stationary ergodic random environment whose marginal distribution is that of the random vector $\mathcal E=(\theta,\lambda,p,d)$ with $\theta>\lambda$\ a.s.\ and $d\in\mathbb N$, $d\ge2$.  Assume
\[
\E[|\log p|]<\infty,\qquad
\E[|\log(1+\lambda/\theta)|]<\infty,
\]
and that $\E\bigl[\bigl|\log\bigl(\p(\xi_{1,0}>0\mid\mathcal E_0)\bigr)\bigr|\bigr]<+\infty$ (which is automatically satisfied for this model under the previous hypotheses).  Then, for every dispersion mechanism $\Delta^{(j)}$, $j=1,2,3,4$:
\begin{itemize}
  \item if
  $\displaystyle \E[\log p] < \E\!\left[\log\frac{\theta}{\theta+\lambda}\right]$,
  the process becomes extinct almost surely;
  \item if almost surely
  \[
  p > \frac1d \qquad\text{and}\qquad \frac{\lambda}{\theta} > \frac1{(d-1)p},
  \]
  the process survives with positive probability.
\end{itemize}
\end{theorem}

\begin{proof}
From Lemma~\ref{lem:ordering} we have, pointwise on the environment,
\[
\mu^{(4)}(\mathcal E)\le\mu^{(j)}(\mathcal E)\le\mu^{(1)}(\mathcal E)
\]
for $j=1,2,3,4$, and therefore
\[
\E[\ell_0^{(4)}]\le\E[\ell_0^{(j)}]\le\E[\ell_0^{(1)}].
\]

For the Poissonian–binomial model of Example~\ref{ex:poisson-binomial}, $\mu^{(1)}=p(1+\lambda/\theta)$. Hence $\log\mu^{(1)}=\log p-\log(\theta/(\theta+\lambda))$, and the hypothesis $\E[\log p]<\E[\log(\theta/(\theta+\lambda))]$ is exactly $\E[\log\mu^{(1)}]<0$. By the upper bound, $\E[\log\mu^{(j)}]<0$ for all $j$, and Theorem~\ref{thm:RE-classification} forces almost sure extinction.

Now assume the almost sure inequalities hold. Set $a=\lambda p/\theta$; the hypotheses are exactly $p>1/d$ and $a>1/(d-1)$.  We show that under these conditions $\mu^{(4)}>1$ almost surely. From Lemma~\ref{lem:mean-D4} and
the integral representation derived in Example~\ref{ex:poisson-binomial},
\[
\mu^{(4)}=d-d(d-1)\!\int_0^1\frac{(1-p)t^{d-2}+p\,t^{d-1}}{1+at}\,dt .
\]
As in the proof of Corollary~\ref{cor:VE-survival-sufficient}, the integrand
divided by $1/(1+at)$ is non‑increasing in $p$, and the whole expression is
strictly decreasing in $a$. Therefore, for every realisation,
\[
\mu^{(4)} > d - d(d-1)\!\int_0^1\frac{(1-p)t^{d-2}+p\,t^{d-1}}{1+at}\,dt
        \Bigg|_{p=1/d,\;a=1/(d-1)} .
\]
At the boundary $p=1/d$, $a=1/(d-1)$ the calculation simplifies
dramatically, giving
\[
\mu^{(4)}\big|_{p=1/d,\;a=1/(d-1)}
= d - d(d-1)\cdot\frac1d = 1 .
\]
Thus, under the strict inequalities $p>1/d$ and $a>1/(d-1)$, we have
$\mu^{(4)}>1$ almost surely.  Consequently $\log\mu^{(4)}>0$ a.s., so
$\E[\log\mu^{(4)}]>0$, and by Lemma~\ref{lem:ordering}
$\E[\log\mu^{(j)}]\ge\E[\log\mu^{(4)}]>0$ for all $j$.
Theorem~\ref{thm:RE-classification} yields survival with positive probability
for every dispersion mechanism.
\end{proof}

\subsection{Case studies with Yule--Simon growth}
\label{sec:yule-examples}

In all examples of this subsection the internal growth is Yule with rate $\lambda$, survival after a catastrophe is binomial with probability $p$, and we use the full dispersal mechanism (D1) unless stated otherwise.

\begin{example}[Varying catastrophe intensity]
\label{ex:VE-catastrophe}
We model a situation in which newly founded communities progressively move away from an epicenter of catastrophic events, but where a minimal level of catastrophe intensity remains unavoidable. This is represented by
\[
\theta_n = \lambda + \varepsilon + a_n>0,
\qquad \text{for a \ \ }\varepsilon>0,
\]
where \((a_n)_{n\ge1}\) is a deterministic sequence satisfying
\(a_n\downarrow0\). Hence, the catastrophes become milder as communities move away from the epicenter, but a baseline risk remains.

The varying environment $(\mathcal E_n)$ is so that $\mathcal E_n = (\theta_n,\lambda,p,d)$. 
The induced offspring mean is
\[
\mu_n = p\,\frac{\theta_n}{\theta_n-\lambda}
      = p\,\frac{\lambda+\varepsilon+a_n}{\varepsilon+a_n}.
\]
\end{example}

\begin{proposition}[Survival under decreasing catastrophe intensity]
\label{prop:VE-catastrophe}
In Example~\ref{ex:VE-catastrophe} the process survives with positive probability
whenever
\[
p > \frac{\varepsilon}{\lambda+\varepsilon} .
\]
\end{proposition}

\begin{proof}
As $a_n\to0$,
$\ell_n = \log p + \log(\lambda+\varepsilon+a_n) - \log(\varepsilon+a_n)
      \to \log p + \log(1+\lambda/\varepsilon)$.
If the limit is strictly positive, then $\underline\Lambda>0$ and
Theorem~\ref{thm:VE-classification} yields survival.  The condition can always be
satisfied by choosing $p$ sufficiently close to $1$ or $\varepsilon$ sufficiently
small compared to $\lambda$.
\end{proof}

\begin{example}[Varying fertility under cumulative stress]
\label{ex:VE-fertility}
Repeated catastrophic events may have long-term biological or social effects, reducing fertility across generations. We model long-term biological or social damage caused by repeated catastrophes, leading to a gradual loss of fertility. Let
\[
\lambda_n = \lambda_0\,e^{-\beta n},
\qquad \beta>0,
\]
with \(\theta>\lambda_0\) fixed. Under full dispersion (D1),
\[
\mu_n
=
p\,\frac{\theta}{\theta-\lambda_n},
\qquad
\ell_n
=
\log p + \log\!\left(\frac{\theta}{\theta-\lambda_n}\right).
\]
\end{example}

The next result shows that, even with mild catastrophes, cumulative damage to fertility can dominate and drive extinction.

\begin{proposition}[Extinction by cumulative fertility loss]
\label{prop:VE-fertility}
In Example~\ref{ex:VE-fertility}, if $p<1$ the process becomes extinct almost surely for all dispersal mechanisms $\Delta^{(j)}$, with $j=1,2,3,4$.
\end{proposition}

\begin{proof}
Since \(\lambda_n=\lambda_0 e^{-\beta n}\to0\), one has
$\sum_{n\ge1}\log\!\bigl(\theta/(\theta-\lambda_n)\bigr)<\infty$, and therefore
\[
\ell_n \longrightarrow \log p < 0,
\qquad
\frac{1}{n}\sum_{k=1}^n \ell_k \longrightarrow \log p.
\]
Hence $\overline\Lambda = \log p <0$ whenever $p<1$, and Theorem~\ref{thm:VE-classification} forces extinction.
\end{proof}

Now we consider models with environments that evolve randomly while the internal growth and
catastrophe parameters remain within a safe range ($\theta>\lambda$).

\begin{example}[Correlated catastrophe intensity]
\label{ex:RE-correlated}

We consider a stationary ergodic random environment modeling alternating periods
of recovery and collapse, while remaining uniformly away from the infinite-mean
regime. Fix $\varepsilon>0$ and assume $\theta_0>\lambda+\varepsilon$.

The catastrophe rate evolves as a Markov chain
\[
\theta_{n+1}
=
\begin{cases}
\lambda+\varepsilon + \gamma(\theta_n-(\lambda+\varepsilon)),
& \text{with probability } \rho,\\[1ex]
\theta_0,
& \text{with probability } 1-\rho,
\end{cases}
\qquad \gamma\in(0,1),\ \rho\in(0,1).
\]
Then $\theta_n\in[\lambda+\varepsilon,\theta_0]$ for all $n$. Under full dispersion~(D1),
\[
\mu_n
=
p\,\frac{\theta_n}{\theta_n-\lambda},
\qquad
\ell_n
=
\log p+\log\!\left(\frac{\theta_n}{\theta_n-\lambda}\right).
\]

Note that the parameter $\rho$ does not enter the expression of $\ell_n$ directly,
but affects the survival criterion through the stationary distribution of the
environment process $(\theta_n)$ (see Proposition~\ref{prop:RE-correlated}).
\end{example}

\begin{proposition}[Survival condition for correlated catastrophes]
\label{prop:RE-correlated}
In Example~\ref{ex:RE-correlated} the process survives with positive probability
iff
\[
L_{corr}:=(1-\rho)\sum_{k\ge0}\rho^k
\log\!\left(
\frac{\lambda+\varepsilon+(\theta_0-(\lambda+\varepsilon))\gamma^k}
{\varepsilon+(\theta_0-(\lambda+\varepsilon))\gamma^k}
\right) > \log\left(\frac{1}{p}\right),
\]
and becomes extinct a.s.\ iff $L_{corr}\le\log(1/p)$.
\end{proposition}

\begin{proof}
Observe that the second integrability condition in Theorem~\ref{thm:RE-classification} is automatically satisfied because $\p(\xi_{1,0}>0\mid\mathcal E_0)\ge p>0$ almost surely and $p$ is constant.

The Markov chain $(\theta_n)$ is irreducible and aperiodic on the compact
interval $[\lambda+\varepsilon,\theta_0]$, hence admits a unique stationary
distribution $\pi$.
Consequently, $(\ell_n)$ is stationary and ergodic and
\[
\lim_{n\to\infty}\frac{1}{n}\sum_{k=1}^n \ell_k
=
\mathbb E_\pi[\ell_0]
\qquad\text{a.s.}
\]
Observe that the stationary distribution of $(\theta_n)$ is the law of
\[
\Theta=\lambda+\varepsilon+(\theta_0-(\lambda+\varepsilon))\gamma^K,
\]
where $K$ is geometric with parameter $1-\rho$.
Consequently,
\[
\mathbb E[\ell_0]
=
\log p
+
\sum_{k\ge0}(1-\rho)\rho^k
\log\!\left(
\frac{\lambda+\varepsilon+(\theta_0-(\lambda+\varepsilon))\gamma^k}
{\varepsilon+(\theta_0-(\lambda+\varepsilon))\gamma^k}
\right).
\]
The claim follows from
Theorem~\ref{thm:RE-classification}.
\end{proof}

Therefore, recovery phases correspond to gradual decreases in catastrophe intensity, but a
baseline level of catastrophic risk remains unavoidable. Survival is driven by
extended periods near the lower bound $\lambda+\varepsilon$, while collapses
reset the system to highly catastrophic states.

\begin{example}[Random environment driven by adaptive survival] \label{ex:RE-adaptive}

We consider a random environment acting on the survival probability,
modeling adaptive or learned resilience. Let \((p_n)\) be a Markov chain on
\((0,1)\) defined by
\[
p_{n+1}
=
\min\{1,\; a\,p_n + (1-a)\,\psi_{n+1}\},
\qquad a\in(0,1),
\]
where \((\psi_n)\) is an i.i.d.\ sequence with a continuous distribution on
\((0,1)\).
Assume \(\theta>\lambda\), and keep \(\lambda\), \(\theta\), and \(d\) fixed. Under dispersion D1,
\[
\mu_n
=
p_n\,\frac{\theta}{\theta-\lambda},
\qquad
\ell_n
=
\log p_n + \log\!\left(\frac{\theta}{\theta-\lambda}\right).
\]
\end{example}

We further develop some properties of the Markov chain $(p_n)$ before turning to the survival and extinction results.

\begin{proposition}[Control of the adaptive log-survival term]
\label{prop:logp-control}
In Example~\ref{ex:RE-adaptive}, assume that $\psi_0$ has a continuous distribution on $(0,1)$ and
that $\mathbb E[|\log \psi_0|]<+\infty$. Let $\pi_p$ denote the invariant
distribution of the Markov chain $(p_n)$ defined by
\[
p_{n+1}=\min\{1,\;a p_n+(1-a)\psi_{n+1}\},\qquad a\in(0,1).
\]
Then:
\begin{enumerate}
\item[\textnormal{(i)}]
$\log p_0\in L^1(\pi_p)$ and
\[
-\infty < \mathbb E_{\pi_p}[\log p_0] < 0.
\]
\item[\textnormal{(ii)}]
The map $a\mapsto \mathbb E_{\pi_p}[\log p_0]$ is continuous and nondecreasing on
$(0,1)$.
\item[\textnormal{(iii)}]
The following bounds hold:
\[
\mathbb E[\log \psi_0]
\;\le\;
\mathbb E_{\pi_p}[\log p_0]
\;\le\;
\log\!\bigl(a+(1-a)\mathbb E[\psi_0]\bigr).
\]
\end{enumerate}
\end{proposition}
\begin{proof}
We begin proving item (i). Since $p_n\in(0,1]$ almost surely, $\log p_n\le0$.
Moreover, from the recursion,
\[
p_{n+1}\ge (1-a)\psi_{n+1},
\]
and hence
\[
\log p_{n+1}\ge \log(1-a)+\log\psi_{n+1}.
\]
Taking expectations under stationarity and using $\mathbb E[|\log\psi_0|]<\infty$ yields $\mathbb E_{\pi_p}[|\log p_0|]<\infty$ and finiteness. Strict negativity follows from $\mathbb P(p_0<1)>0$.

We now turn to item (ii). For fixed $\psi$, the map $p\mapsto \min\{1,ap+(1-a)\psi\}$ is increasing in $a$. By standard monotone coupling arguments for Markov chains, the stationary distribution $\pi_p^{(a)}$ is stochastically increasing in $a$. Since $\log p$ is
increasing, $\mathbb E_{\pi_p}[\log p_0]$ is nondecreasing in $a$. Continuity follows from dominated convergence.

It remains to verify (iii). The lower bound follows from Jensen’s inequality applied to the inequality $p_{n+1}\ge(1-a)\psi_{n+1}$. For the upper bound, taking expectations in the stationary equation and using concavity of $\log$ yields
\[
\mathbb E_{\pi_p}[\log p_0]
=
\mathbb E_{\pi_p}\!\left[\log\bigl(a p_0+(1-a)\psi_1\bigr)\right]
\le
\log\!\bigl(a+(1-a)\mathbb E[\psi_0]\bigr).
\]
\end{proof}

The next proposition shows that, for Example \ref{ex:RE-adaptive}, survival depends not only on the average resilience but also on its variability. Rare generations with very small \(p_n\) may force extinction even when the average survival probability is large.

\begin{proposition}[Survival condition for adaptive survival]
\label{prop:RE-adaptive}
In Example~\ref{ex:RE-adaptive} the process survives with positive probability
iff
\[
\E_{\pi_p}[\log p_0] > \log\frac{\theta-\lambda}{\theta},
\]
and becomes extinct a.s.\ iff $\E_{\pi_p}[\log p_0] \le \log\frac{\theta-\lambda}{\theta}$, where $\pi_p$ is the
stationary distribution of $(p_n)$.
\end{proposition}

\begin{proof}
Note that the second integrability condition of Theorem \ref{thm:RE-classification} holds because Proposition~\ref{prop:logp-control} implies $\E_{\pi_p}[|\log p_0|]<\infty$, and we have $\p(\xi_{1,0}>0\mid\mathcal E_0)\ge p_0$ almost surely.

The chain \((p_n)\) admits a unique invariant distribution \(\pi_p\) and is
ergodic. By Birkhoff’s ergodic theorem,
\[
\lim_{n\to\infty}\frac{1}{n}\sum_{k=1}^n \ell_k
=
\mathbb E_{\pi_p}[\log p_0]
+
\log\!\left(\frac{\theta}{\theta-\lambda}\right)
\qquad\text{a.s.}
\]
Note that the parameters $a$ and the law of $(\psi_n)$ do not enter the expression
of $\mu_n$ directly, but determine the stationary distribution $\pi_p$ of the
environment process $(p_n)$ and hence the value of
$\mathbb E_{\pi_p}[\log p_0]$.
\end{proof}

\begin{remark}
Proposition~\ref{prop:RE-adaptive} shows that survival is determined by the
geometric mean of the survival probability, $\exp(\E_{\pi_p}[\log p_0])$, rather
than its arithmetic mean $\E_{\pi_p}[p_0]$.  By Jensen's inequality,
$\E_{\pi_p}[\log p_0] \le \log \E_{\pi_p}[p_0]$, with strict inequality whenever
$p_0$ is non-degenerate.  Hence, highly variable survival probabilities can lead
to extinction even when the average probability of surviving a catastrophe is
arbitrarily close to $1$.
\end{remark}

\subsection*{Acknowledgements}
The authors used a Large Language Model from DeepSeek to improve the clarity and linguistic flow of the manuscript. All mathematical content and conclusions remain the sole responsibility of the authors.

\bibliographystyle{abbrvnat}
\bibliography{references}

\end{document}